\documentclass[12pt, reqno]{amsart}

\usepackage[utf8]{inputenc}
\usepackage[T1]{fontenc}
\usepackage{tikz,tikz-3dplot}
\usepackage{amsmath,amsfonts,amssymb,amsmath,latexsym,mathrsfs}
\usepackage{shuffle}
\usepackage[all]{xy}
\usepackage{stackengine}
 
\usepackage{enumerate}

\usepackage{hyperref}
 
\usepackage{tabu}
\usepackage{tikz-cd} 

\usepackage{comment}
\usepackage{url}
\usepackage{tikz}

\usepackage{xfrac}

\newtheorem{theorem}{Theorem}[section]

\newtheorem{proposition}[theorem]{Proposition}

\newtheorem{corollary}[theorem]{Corollary}

\newtheorem{remark}[theorem]{Remark}

\newtheorem{lemma}[theorem]{Lemma}

\newcommand{\Z}{{\mathbb Z}}
\newcommand{\ZZ}{ {\mathbb{Z}/2\mathbb{Z}}}

\newcommand{\Q}{{\mathbb Q}}

\newcommand{\R}{ {\mathbb{R}} }
 \newcommand{\RP}{ {\mathbb{R}\mathbb{P}} }
\begin{document}

\title{Homology of degenerate real projective quadrics}
\keywords{Real projective quadrics, Homology}

\author{Mohamad \textsc{Maassarani} }
\maketitle
\begin{abstract}
Homology of non degenerate real projective quadrics was studied by Steenrod and Tucker. We Compute the rational and the $\ZZ$ homology of degenerate real projective quadrics. This allows to determine the integer homology of these quadrics. 
\end{abstract}
\section*{Introduction and main results}
A quadratic form on $\R^n$ defines a real projective quadric in $\RP^{n-1}$. The homology groups of non degenerate quadrics, was studied by Steenrod and Tucker, in \cite{ST41}. They show that a non degenerate quadric is a sphere bundle over a projective space, and that the homology over $\ZZ$ of the quadric is the tensor product of the homology of the base and the fiber. In this document, we consider the homology of degenerate quadrics. We compute the rational and $\ZZ$-homology of these quadrics. This allows to determine the integer homology of degenerate quadrics.\\\\
A real projective quadric is homeomorphic by a projective transformation to $Q_{p,q}^n$ the real locus in $\RP^{n-1}$ of the quadratic form $\mathbf{q}(x_1,\dots,x_n)=x_1^2+\cdots +x_p^2-x_{p+1}^2-\cdots -x_{p+q}^2$ for some $p,q$ ($p+q\leq n$). Hence, in this text we only consider $Q_{p,q}^n$ and the computations are somtimes done for $q\geq p$ since $Q_{p,q}^n$ is homeomorphic to $Q_{q,p}^n$. We note also that $Q_{0,q}^n$ is either empty either homeomorphic to a projective space.\\\\
The paper is divided into 4 sections :\\

In the first section, we compute the integer homology of a join $X\star Y$ of two spaces $X$ and $Y$ in terms of the homology of $X$ and $Y$, under the assumption $Tor(H_i(X,\Z),H_j(Y,\Z))=0$ for all $i,j$. We also consider the induced map in homology $H_*(X\star Y,\Z) \to H_*(X\star Y,\Z)$ of $f\star g$ constructed out of two continuous maps $f:X\to X$ and $g:Y\to Y$. Results of section 1 are used in section 2. \\

In the second section, we show that for, $p,q,n-p-q\geq 1$, the space $X_{p,q}^n$ obtained by intersecting the locus of $\mathbf{q}$ in $\R^n$ with the $n-1$-dimensional unit sphere $S^{n-1}\subset \R^n$, is a $2$-sheeted cover of $Q_{p,q}^n$. We prove that $X_{p,q}^n$ is homeomorphic to the join $(S^{p-1}\times S^{q-1})\star S^{n-p-q-1}$ and compute its  integer and rational homology. We then compute the rational homology of $Q_{p,q}^n$ as an invariant vector subspace of the rational homology of $X_{p,q}^n$.\\

Section $3$, is devoted to the $\ZZ$-homology groups of $Q_{p,q}^n$. The projectivisation $D$ of the kernel of $\mathbf{q}$ is homeomorphic to $\RP^{n-p-q-1}$, $Q_{p,q}^n\setminus D$ is a real vector bundle of rank $n-p-q$ over $Q_{p,q}^{p+q}$ and the quotient space $Q_{p,q}^n/D$ is homeomorphic to the Thom space of $Q_{p,q}^n\setminus D$. We derive essentially from the exact sequence of the pair $(Q_{p,q}^n,D)$ and the Thom isomorphism that $H_{n-p-q+k}(Q_{p,q}^n,\ZZ)$ is homeomorphic to $H_k(Q_{p,q}^{p+q},\ZZ)$ for $k>0$. This gives the homology groups in "high" degree, since the homology groups of  $Q_{p,q}^{p+q}$ are known from \cite{ST41}. The "lower" $\ZZ$-homology groups are computed using the exact sequence of the cover $X_{p,q}^n\to Q_{p,q}^n$.\\

In the last section, we show that the integer homology groups $H_k(Q_{p,q}^n,\Z)$ of $Q_{p,q}^n$ are of the form $\Z^{m_k}\oplus \ZZ^{l_k}$ where $m_l$ and $l_k$ are determined by the Betti numbers of $Q_{p,q}^n$ over $\Q$ and over $\ZZ$. We compute the integer homology of $Q_{p,q}^n$ for the case $q>p>1$, $n>p+q+1$ and where $p,q$ and $n$ are even. The other cases can also be computed; we do not consider them to avoid more computations.   

\section{On joins}\label{S1}
In the first subsection, we compute the integer homology of the join $X\star Y$ of two spaces $X$ and $Y$ in terms of the homology of $X$ and $Y$, under the assumption $Tor(H_i(X,\Z),H_j(Y,\Z))=0$ for all $i,j$. In the second subsection, under the same assumption, we consider the induced map in homology $H_*(X\star Y,\Z) \to H_*(X\star Y,\Z)$ of $f\star g$ constructed out of two continuous maps $f:X\to X$ and $g:Y\to Y$. \\\\
In this section we work with integer homology and we will write $H_*(-)$ for $H_*(-,\Z)$.  Let $X$ and $Y$ be two topological spaces, the join $X\star Y$ of $X$ and $Y$ is the quotient space defined by :

$$X\star Y=X \times Y \times [0,1] / (x,y,0)\sim (x',y,0), (x,y,1) \sim (x,y',1) $$

\subsection{Homology of joins}
In this subsection, we assume that $Tor(H_i(X),H_j(Y))=0$ for all $i,j$. \\\\
A join is path connected. Hence, $H_0(X\star Y)=\Z$. We will compute the higher homology groups using the Mayer-Vietoris sequence. Take :
$$A=  X \times Y \times [0,\frac{3}{4}] / (x,y,0)\sim (x',y,0)\quad \text{and} \quad B= X \times Y \times [\frac{1}{4},1] / (x,y,1) \sim (x,y',1).$$
The interiors of $A$ and $B$ cover $X\star Y$ and $A\cap B= X\times Y\times [\frac{1}{4},\frac{3}{4}]$. By the Mayer-Vietoris sequence, we have an exact sequence :
$$\cdots \to H_{n+1}(X\star Y) \overset{\partial}{\to} H_n(A\cap B) \overset{i^A\oplus i^B}{\to} H_n(A) \oplus H_n(B) \overset{i_A-i_B}{\to} H_n(X\star Y) \overset{\partial}{\to} H_{n-1}(A\cap B) \to \cdots ,$$ 
where $i^A$ and $i^B$ are respectively the maps induce by the inclusions $A\cap B \to A$ and $A\cap B \to B$ and $i_A$ and $i_B$ are respectivle induced by the inclusions $A\to X\star Y$ , $B\to X\star Y$.
The space $A\cap B$ is homotopy equivalent to $X\times Y$, $A$ and $B$ are respectively homeomorphic to $C_X\times Y$ and $X\times C_Y$ ($C_Z$ is the cone of $Z$). We assume that $Tor(H_i(X),H_j(Y))=0$ for all $i,j$. Morover $C_Z$ is a contractible space, hence by naturality of the Kunneth isomorphism, we get a commutative diagram :

$$\begin{tikzcd}
H_*(A\cap B) \arrow{r}{i^A\oplus i^B} & H_* (A)\oplus H_*(B)  \\%
H_*(X)\otimes H_*(Y)  \arrow[swap]{u}{\simeq} \arrow{r}{l}& H_*(C_X)\otimes H_*(Y) \oplus H_*(X) \otimes H_*(C_Y)\arrow{u}{\simeq}
\end{tikzcd} (D1)$$
 where $l=i_X\otimes \mathrm{id} _{H_* (Y)}\oplus \mathrm{id}_{H_*(X)} \otimes i_Y$ with $i_Z$ denoting the morphism induced by the inclusion $Z\to C_Z$. This will be used to descibe the homology of the join $X\star Y$.
\begin{proposition}
Denote by $l_k$ the degree $k$ part of $l$. The morphism $l_k$ is surjective for $k\geq 1$.
\end{proposition}
\begin{proof}
Since $C_Z$ is contractible as mentioned before $H_*(C_Z)=H_0(C_Z)\simeq \Z$, $i_Z : H_*(Z)\to H_*(C_Z)$ is surjective and the kernel of $i_Z$ contains the groups $H_k(Z)$ for $k\geq 1$ . Denote by $1_{C_Z}$ a generater of $H_*(C_Z)$ as a $\Z$-module and by $1_z$ an element of $H_0(Z)$ such that $i_Z(1_z)=1_{C_Z}$. Now let $a_X$ and $a_Y$ be two elements of degree $k\geq 1$ of $H_*(X)$ and $H_*(Y)$ respectively. We have that : $$l( 1_x\otimes a_Y +a_X \otimes 1_y)=1_{C_X}\otimes a_y+ a_X\otimes 1_{C_Y}.$$
This proves that $l$ is surjective in degree $k\geq 1$.
 \end{proof}
\begin{theorem}\label{Hjo} If $Tor(H_i(X),H_j(Y))=0$ for all $i,j$, then for $n\geq 1$, the group $H_n(X\star Y)$ is isomorphic to the degree $n-1$ par of $Ker(i_X)\otimes Ker(i_Y)$. 
\begin{itemize}
\item[1)]For $n>1$, the degree $n-1$ part of $Ker(i_X)\otimes Ker(i_Y)$ is :
$$\underset{k,m>0}{\underset{k+m=n-1}{\oplus}} H_k(X)\otimes H_m(Y) \oplus Ker((i_X)_0)\otimes H_{n-1}(Y) \oplus H_{n-1}(X) \otimes Ker((i_Y)_0), \quad  $$
where $(i_Z)_0$ is the degree zero part of the morphism $i_Z$ induced by the inclusion $Z\to C_Z$.
\item[2)]For $n=1$, the degree $n-1$ part of $Ker(i_X)\otimes Ker(i_Y)$ is :
$$ Ker((i_X)_0) \otimes Ker((i_Y)_0),$$
where $(i_Z)_0$ is the degree zero part of the morphism $i_Z$ induced by the inclusion $Z\to C_Z$.
 \end{itemize}
\end{theorem}
\begin{proof}
By the previous proposition and diagram $(D1)$ the morphism $i^A\oplus i^B$ is surjective in degree $n\geq 1$. Hence for $n\geq1$, $\partial : H_n(X\star Y) \to H_{n-1}(A\cap B)$ of the Mayer-Vietoris sequence is injective and $H_n(X\star Y)$ is isomorphic to its image under $\partial$. But the image of $\partial : H_n(X\star Y) \to H_{n-1}(A\cap B)$ is exactly the kernel of the degree $n-1$ part of $i^A\oplus i^B$. Hence we have an isomophism by diagram $(D1)$ between $H_n(X\star Y)$ and the kernel of $l_{n-1}$ (the degree $n-1$ part of $l$). The kernel of $l_{n-1}$ is the degree $n-1$ part of $Ker(i_X)\otimes Ker(i_Y)$. The theorem follows from the fact that $Ker(i_Z)=Ker((i_Z)_0)\oplus \underset{k>0}{\oplus}{H_k(Z)}$. 
\end{proof}
\begin{remark}
$C_Z$ is path connected and hence $Ker((i_Z)_0)$ can be determined from $H_0(Z)$.
\end{remark}
\subsection{Join of maps}
Let $f:X\to X$ and $g:Y\to Y$ be continuous maps. The map $f\times g \times\mathrm{id}_{[0,1]}$ induces a unique map $f\star g : X\star Y \to X \star Y$. Let $A$ and $B$ be as in the previous section. We have $f\star g(A)\subset A, f\star g(B)\subset B$ and hence $f\star g(A\cap B)\subset A\cap B$. Hence, since the Mayer-Vietoris sequence is natural (\cite{M91}), we have a commutative diagram : 
 $$\begin{tikzcd}
\cdots \arrow{r}& H_{n+1}(X\star Y) \arrow{r}{\partial} \arrow{d}{(f\star g)_*}& H_n(A\cap B) \arrow{r}{i^A\oplus i^B}\arrow{d}{(f\star g)_*}& H_n(A) \oplus H_n(B)  \arrow{r}\arrow{d}{(f\star g)_*} & \cdots \\
\cdots \arrow{r}& H_{n+1}(X\star Y) \arrow{r}{\partial}& H_n(A\cap B) \arrow{r}{i^A\oplus i^B}& H_n(A) \oplus H_n(B)  \arrow{r}&  \cdots 
\end{tikzcd}$$
\begin{proposition}\label{Mjo}
Assume that $Tor(H_i(X),H_j(Y))=0$ for all $i,j$ and denote by $ (Ker(i_X)\otimes Ker(i_Y))_n $ the degree $n$ part of $ Ker(i_X)\otimes Ker(i_Y) $. For $n\geq 0$, we have a commutative diagram :
 $$\begin{tikzcd}
 H_{n+1}(X\star Y) \arrow{r}{\simeq} \arrow{d}{(f\star g)_*}& (Ker(i_X)\otimes Ker(i_Y))_n \arrow{d}{f_*\otimes g_*}  \\ 
H_{n+1}(X\star Y) \arrow{r}{\simeq}& (Ker(i_X)\otimes Ker(i_Y))_n 
\end{tikzcd} $$
\end{proposition}
\begin{proof}
The spaces $C_X$ and $C_Y$ are contractible. Hence, $f_*(Ker(i_X))\subset Ker(i_X)$ and $g_*(Ker(i_Y))\subset Ker(i_Y)$ and we have a commutative diagram :
$$\begin{tikzcd}
 H_*(X)\otimes H_*(Y) \arrow{d}{f_*\otimes g_*} &\arrow{l}Ker(i_X)\otimes Ker(i_Y) \arrow{d}{f_*\otimes g_*}  \\ 
 H_*(X)\otimes H_*(Y) &\arrow{l}Ker(i_X)\otimes Ker(i_Y) 
\end{tikzcd} $$
By the diagram preceeding the proposition, we  have a commutative diagram :

$$\begin{tikzcd}
H_{*+1}(X\star Y) \arrow{r}{\partial} \arrow{d}{(f\star g)_*}&H_*(A\cap B)  \arrow{d}{(f\star g)_*}  \\ 
H_{*+1}(X\star Y) \arrow{r}{\partial} &H_*(A\cap B) 
\end{tikzcd} $$

Now the map $f\star g $ over $A\cap B$ is nothing else but $f\times g \times \mathrm{id}_{[\frac{1}{4},\frac{3}{4}]}$ and the inclusion $X\times Y\to X\times Y\times \{\frac{1}{2}\} \subset A\cap B$ is a homotopy equivalence. Hence, we can join the last to diagram to get :
 $$\begin{tikzcd}
H_{*+1}(X\star Y) \arrow{r}{\partial} \arrow{d}{(f\star g)_*}&H_*(A\cap B)  \arrow{d}{(f\star g)_*}& \arrow{l}{\simeq} H_*(X)\otimes H_*(Y) \arrow{d}{f_*\otimes g_*} &\arrow{l}Ker(i_X)\otimes Ker(i_Y) \arrow{d}{f_*\otimes g_*}  \\ 
H_{*+1}(X\star Y) \arrow{r}{\partial} &H_*(A\cap B)&\arrow{l}{\simeq} H_*(X)\otimes H_*(Y) &\arrow{l}Ker(i_X)\otimes Ker(i_Y) 
\end{tikzcd} $$
By the proof of the previous theorem $\partial$ is injective and the image of $H_{*+1}(X\star Y)$ under $\partial$ correspond to the (injective) image of $Ker(i_X)\otimes Ker(i_Y)$ under the $2$ left arrows. This proves the proposition.
\end{proof}

\section{Double cover and  rational homology of degenerate quadrics}
For $p,q \geq 1$ and $n\geq p+q$, we denote by $Q_{p,q}^n$ is the real locus in $\RP^{n-1}$ of the quadratic form $\mathbf{q} : \R^n \to \R,\mathbf{q}(x_1,\dots,x_n)=x_1^2+\cdots +x_p^2-x_{p+1}^2-\cdots -x_{p+q}^2$. The real locus in $\RP^{n-1}$ of any quadratic form of signature $(p,q)$ of $\R^n$, is homeomorphic via a projective trasformation to $Q_{p,q}^n$.
If $p$ or $q$ are zero, $Q_{p,q}^n$ is homeomorphic to a projective space (eventually $\emptyset$). We assume in the following that $p,q,n-p-q\geq 1$.\\\\
In this section, we consider, in the first subsection, a $2$-sheeted cover of $Q_{p,q}^n$ and show that it is homeomorphic to a join. We describe the deck transformation group of this cover.
 In the second subsection, we compute the integer and rational homology of the $2$-sheeted cover using results of section \ref{S1}. We then compute the rational homology of $Q_{p,q}^n$ as an invariant vector subspace of the rational homology of the cover.
\subsection{The double cover and the antipode}  Denote by $X_{p,q}^n$ the intersection : 
$$X_{p,q}^n =\{x\in \R^n \vert \mathbf{q}(x)=0\}\cap S^{n-1},$$
where $S^{n-1}$ denotes the unit sphere of $\R^n$. 
Let $\vert \vert \cdot \vert\vert$ denote the euclidian norm of $\R^k$. By identifying $\R^n$ to $\R^p\times \R^q\times \R^{n-p-q}$, $X_{p,q}^n$ is identified to :
$$ X_{p,q}^n=\{(x,y,z) \in \R^p\times \R^q\times \R^{n-p-q} \vert \ \   \vert \vert x \vert\vert=\vert \vert y \vert\vert , \vert \vert x \vert\vert^2+\vert \vert y \vert\vert^2+\vert \vert z \vert\vert^2=1\}.$$
Define the continuous map $f:S^{p-1}\times S^{q-1}\times S^{n-p-q-1}\times [0,1] \to X_{p,q}^n$ given by :
$$f(x,y,z,t)=(\frac{t}{\sqrt{2}}x,\frac{t}{\sqrt{2}}y ,(1-t)z),$$
for $(x,y,z,t)\in S^{p-1}\times S^{q-1}\times S^{n-p-q-1}\times[0,1]$.

\begin{proposition}
\begin{itemize}
\item[1)] The map $f$  defined above is surjective.
\item[2)] $f(x,y,z,t)=f(x',y',z',t')$ if and only if $t=0$ and $(x,y)=(x',y')$ or $t=1$ and $z=z'$.
\end{itemize}
\end{proposition}
\begin{proof}
This can be checked using the identification $$X_{p,q}^n=\{(x,y,z) \in \R^p\times \R^q\times \R^{n-p-q} \vert \ \   \vert \vert x \vert\vert=\vert \vert y \vert\vert , \vert \vert x \vert\vert^2+\vert \vert y \vert\vert^2+\vert \vert z \vert\vert^2=1\},$$ and the definition of $f$.
\end{proof}
\begin{proposition}\label{jo}
The map $f$ of the previous proposition induces an homeomorphism $\bar{f}: (S^{p-1}\times S^{q-1})\star S^{n-p-q-1}\to X_{p,q}^n$.
\end{proposition}
\begin{proof}
Point $2)$ of the last proposition implies that $f$ induces a continuous injective map $\bar{f}: (S^{p-1}\times S^{q-1})\star S^{n-p-q-1}\to X_{p,q}^n$. With point $1)$ of the same proposition, we get that $\bar{f}$ is a continuous bijection. Since both spaces are compact, we deduce that $\bar{f}$ is a homeomorphism.
\end{proof}
Since a join of spaces is path connected we have :
\begin{corollary}
$X_{p,q}^n$ is path connected.
\end{corollary}
Let $a_{n-1}:S^{n-1} \to S^{n-1}$ be the antipode (involution mapping $x$ to $-x$) and $\pi$ be the projection from $S^{n-1}\to \RP^{n-1}$. $\pi$ is a non trivial $2$-sheeted covering map with deck transformation group $\{1,a_{n-1}\}$. Since $\mathbf{q}(x)=\mathbf{q}(-x)$, $X_{p,q}^n$ is stable under the antipod and : 
\begin{proposition}
The map $X_{p,q}^n\overset{\pi}{\to}{Q_{p,q}^n}$ is a non trivial $2$-sheeted covering map with deck transformation group $\{1,a_{n-1}\}$ where $a_{n-1}$ is the antipod.
\end{proposition} 
\begin{proof}
The map $X_{p,q}^n\overset{\pi}{\to}{Q_{p,q}^n}$ induces a continuous bijection between $X_{p,q}^n/\{1,a_{n-1}\}$ and $Q_{n,p}$. Since both spaces are compact the induced map is a homeomorphism. This proves that  $X_{p,q}^n\overset{\pi}{\to}{Q_{p,q}^n}$ is a $2$-sheeted cover because $X_{p,q}^n\to X_{p,q}^n/\{1,a_{n-1}\}$ is a covering map. Finally, the cover is not trivial since $X_{p,q}^n$ is path connected.
\end{proof}
\begin{proposition}\label{ant}
Let $\bar{f}$ be the homeomorphism of proposition \ref{jo}. For $w\in X_{p,q}^n$, we have $$ \bar{f} ((a_{p-1}\times a_{q-1}) \star a_{n-p-q-1})\bar{f}^{-1}(w)=a_{n-1}(w),$$
where $a_k$ denotes the antipode of $S^k$. 
\end{proposition}
\begin{proof}
We will prove that $$ ((a_{p-1}\times a_{q-1}) \star a_{n-p-q-1})(v)= \bar{f}^{-1}a_{n-1}\bar{f}(v),$$ for $v \in (S^{p-1}\times S^{q-1})\star S^{n-p-q-1}$. We only need to prove the equality for $v\in (S^{p-1}\times S^{q-1})\times S^{n-p-q-1} \times ]0,1[$, since it is a dense subset of the join. Take $v=(x,y,z,t)\in  S^{p-1}\times S^{q-1}\times S^{n-p-q-1} \times ]0,1[$. We have by definition that :
$$ ((a_{p-1}\times a_{q-1}) \star a_{n-p-q-1})(x,y,z,t)=(-x,-y,-z,t). $$
Indeed, $ ((a_{p-1}\times a_{q-1}) \star a_{n-p-q-1})$ is induced by $ ((a_{p-1}\times a_{q-1}) \times  a_{n-p-q-1})\times \mathrm{id}_{[0,1]}$. On the other hand :
\begin{align*}
 \bar{f}^{-1}a_{n-1}\bar{f}(x,y,z,t)&=\bar{f}^{-1}a_{n-1}(\frac{t}{\sqrt{2}}x,\frac{t}{\sqrt{2}}y ,(1-t)z)\\
&=\bar{f}^{-1}(-\frac{t}{\sqrt{2}}x,-\frac{t}{\sqrt{2}}y ,-(1-t)z)\\
&=(-x,-y,-z,t).
\end{align*}
We have proved the proposition.
\end{proof} 
\subsection{Homology of the double cover and rational homology of degenerate quadrics}
We will compute homology groups. The computations are divided into six cases.

\subsubsection{The case $p,q,n-p-q>1$}\label{S221}
We will assume that $p,q,n-p-q>1$. By the result of the previous section $X_{p,q}^n$ is homeomorphic to the join $(S^{p-1}\times S^{q-1})\star S^{n-p-q-1}$. In particular, $H_0(X_{p,q}^n,\Z)=\Z$. We have that $Tor(H_i(S^{p-1}\times S^{q-1},\Z),H_j(S^{n-p-q-1},\Z))=0$ for all $i,j$. Hence, by theorem \ref{Hjo} :

\begin{equation}\label{iso}
 H_{*+1}(X_{p,q}^n,\Z) \simeq Ker(i_{S^{p-1}\times S^{q-1}})\otimes Ker(i_{S^{n-p-q-1}}),
\end{equation}
where $i_Z$ is the map induced in homology by the inclusion of the space $Z$ into the cone $C_Z$ which is a contractible space. For $i>0$ , denote by $1$ a generator of $H_0(S^i,\Z)$ and by $b_i$ a generator of $H_i(S^i,\Z)$
The spaces $S^{p-1}\times S^{q-1}$ and $S^{n-p-q-1}$ are path connected ($ p,q,n-p-q >1$) and hence by using the Kunneth isomorphism $H_*(S^{p-1}\times S^{q-1},\Z)\simeq H_*(S^{p-1},\Z)\otimes H_*(S^{q-1},\Z),$ we have that :
\begin{itemize}
\item[1)] $Ker(i_{S^{p-1}\times S^{q-1}})$ is generated as a $\Z$-module by $b_{p-1}\otimes 1, 1\otimes b_{q-1}$ and $ b_{p-1}\otimes b_{q-1}$.
\item[2)] $Ker(i_{S^{n-p-q-1}})$ is generated as a $\Z$-module by $b_{n-p-q-1}$.
\end{itemize}
In particular, $H_{*+1}(X_{p,q}^n,\Z)$ viewed as a subspace of $$H_*(S^{p-1},\Z)\otimes H_*(S^{q-1},\Z)\otimes H_*(S^{n-p-q-1},\Z)$$ is generated as a $\Z$-module by : 
$$b_{p-1}\otimes 1\otimes b_{n-p-q-1}, 1\otimes b_{q-1}\otimes b_{n-p-q-1}\ \text{and} \  b_{p-1}\otimes b_{q-1}\otimes b_{n-p-q-1} .$$
It follows that :
 
\begin{proposition}If $p\neq q$ then :
\begin{itemize}
\item[1)]The integer homology of $X_{p,q}^n$ is given by :
 \begin{equation*}
      H_i(X_{p,q}^n,\Z) \simeq 
        \begin{cases}
            \Z& \text{if $i =0,n-p-1, n-q-1,n-2$} \\
           
            0 & \text{otherwise}
        \end{cases}
    \end{equation*}
 \item[2)]The rational homology of $X_{p,q}^n$ is given by :
\begin{equation*}
      H_i(X_{p,q}^n,\mathbb{Q}) \simeq 
        \begin{cases}
            \mathbb{Q}& \text{if $i =0,n-p-1, n-q-1,n-2$} \\
           
            0 & \text{otherwise}
        \end{cases}
    \end{equation*}

\end{itemize}
\end{proposition}
\begin{proposition}If $p=q$ then :
\begin{itemize}
\item[1)]The integer homology of $X_{p,q}^n$ is given by :
 \begin{equation*}
      H_i(X_{p,q}^n,\Z) \simeq 
        \begin{cases}
            \Z& \text{if $i =0,n-2$} \\
           \Z^2 &\text{if $i=n-p-1=n-q-1$}\\
            0 & \text{otherwise}
        \end{cases}
    \end{equation*}
 \item[2)]The rational homology of $X_{p,q}^n$ is given by :
\begin{equation*}
      H_i(X_{p,q}^n,\mathbb{Q}) \simeq 
        \begin{cases}
            \mathbb{Q}& \text{if $i =0,n-2$} \\
            \Q^2&  \text{if $i=n-p-1=n-q-1$}\\
            0 & \text{otherwise}
        \end{cases}
    \end{equation*}

\end{itemize}
\end{proposition}
We will denote by $a_i$ the antipod of $S^i$.  
\begin{lemma}
\begin{itemize}
\item[1)] $$(a_{p-1})_*\otimes (a_{q-1})_*\otimes (a_{n-p-q-1})_*(b_{p-1}\otimes 1\otimes b_{n-p-q-1})=(-1)^{n-q}b_{p-1}\otimes 1\otimes b_{n-p-q-1}.$$
\item[2)]$$(a_{p-1})_*\otimes (a_{q-1})_*\otimes (a_{n-p-q-1})_*(1\otimes b_{q-1}\otimes b_{n-p-q-1})=(-1)^{n-p}1\otimes b_{q-1}\otimes b_{n-p-q-1}.$$
\item[3)]$$(a_{p-1})_*\otimes (a_{q-1})_*\otimes (a_{n-p-q-1})_*(b_{p-1}\otimes b_{q-1}\otimes b_{n-p-q-1})=(-1)^{n}b_{p-1}\otimes b_{q-1}\otimes b_{n-p-q-1}.$$
\end{itemize}
\end{lemma}
\begin{proof}
We use that the degree of $a_k$ is $(-1)^{k+1}$ and that $a_k(1)=1$.
\end{proof}
\begin{theorem}\label{rational}For $p\neq q$, 
$H_0(Q_{p,q}^n,\Q)=\Q$, $H_{n-q-1}(Q_{p,q}^n,\Q)=\Q$ if $n-q$ is even, $H_{n-p-1}(Q_{p,q}^n,\Q)=\Q$ if $n-p$ is even, $H_{n-2}(Q_{p,q}^n,\Q)=\Q$ if $n$ is even and $H_i(Q_{p,q}^n,\Q)=0$ otherwise.
\end{theorem}
\begin{proof}
Since $X_{p,q}^n$ is a nontrivial $2$-sheeted cover of $Q_{p,q}^n$ with deck transformation group $\{1,a_{n-1}\}$, $H_*(Q_{p,q}^n,\Q)$ is isomorphic to the invariant of $H_*(X_{p,q}^n,\Q)$ by $a_{n-1}$. Now identifying $H_*(X_{p,q}^n,\Q)$ to the homology of the join $(S^{p-1}\times S^{q-1})\star S^{n-p-q-1}$ identifies $a_{n-1}$ to $(a_{p-1}\times a_{q-1}) \star a_{n-p-q-1}$ by proposition \ref{ant}. Indentifying the homology of the join $(S^{p-1}\times S^{q-1})\star S^{n-p-q-1}$ using the isomorphism in equation (\ref{iso}) identifies $(a_{p-1}\times a_{q-1}) \star a_{n-p-q-1}$  to $(a_{p-1}\times a_{q-1})_* \otimes (a_{n-p-q-1})_*$ by proposition \ref{Mjo}. Applying the Kunneth isomorphism to pass to element of $H_*(S^{p-1})\otimes H_*(S^{q-1})\otimes H_*(S^{n-p-q-1})$ identifies $H_{*+1}(X_{p,q}^n)$ with three generators $b_{p-1}\otimes 1\otimes b_{n-p-q-1}, 1\otimes b_{q-1}\otimes b_{n-p-q-1}\ \text{and} \  b_{p-1}\otimes b_{q-1}\otimes b_{n-p-q-1} $ and identifies $(a_{p-1}\times a_{q-1})_* \otimes (a_{n-p-q-1})_*$ with $ (a_{p-1})_*\otimes (a_{q-1})_*\otimes (a_{n-p-q-1})_*$. Therfore the theorem follows from the previous lemma.
 \end{proof}
\begin{theorem}
$H_0(Q_{p,p}^n,\Q)=\Q$, $H_{n-p-1}(Q_{p,p}^n,\Q)=\Q^2$ if $n-p$ is even, $H_{n-2}(Q_{p,p}^n,\Q)=\Q$ if $n$ is even and $H_i(Q_{p,p}^n,\Q)=0$ otherwise.
\end{theorem}
\begin{proof}
The proof is exactly the same as for the previous theorem.
\end{proof}

\subsubsection{The case $p,q>1$, $n-p-q=1$}\label{S222}
We consider the case $p,q>1$ and $n=p+q+1$. By the results of the previous section $X_{p,q}^{p+q+1}$ is homeomorphic to the join $(S^{p-1}\times S^{q-1})\star S^{0}$. In particular, $H_0(X_{p,q}^{p+q+1},\Z)=\Z$. We have that $Tor(H_i(S^{p-1}\times S^{q-1},\Z),H_j(S^{0},\Z))=0$ for all $i,j$. Hence, by theorem \ref{Hjo} :

\begin{equation*}
 H_{*+1}(X_{p,q}^n,\Z) \simeq Ker(i_{S^{p-1}\times S^{q-1}})\otimes Ker(i_{S^{0}}),
\end{equation*}
where $i_Z$ is the map induces by the inclusion of the space $Z$ into the cone $C_Z$ which is a contractible space. For $i>0$ , denote by $1$ a generator of $H_0(S^i,\Z)$ and by $b_i$ a generator of $H_i(S^i,\Z)$ and denote by $b_0$ and $b_0'$ two generators of $H_*(S^0,\Z)=H_0(S^0,\Z)\simeq \Z^2$ such that $b_0$ and $b_0'$ correpond to two simplices. 
Using the Kunneth isomorphism $H_*(S^{p-1}\times S^{q-1},\Z)\simeq H_*(S^{p-1},\Z)\otimes H_*(S^{q-1},\Z),$ we have that :
\begin{itemize}
\item[1)] $Ker(i_{S^{p-1}\times S^{q-1}})$ is generated as a $\Z$-module by $b_{p-1}\otimes 1, 1\otimes b_{q-1}$ and $ b_{p-1}\otimes b_{q-1}$.
\item[2)] $Ker(i_{S^0})$ is generated as a $\Z$-module by $b_0-b_0'$.
\end{itemize}
In particular, $H_{*+1}(X_{p,q}^{p+q+1},\Z)$ viewed as a subspace of $$H_*(S^{p-1},\Z)\otimes H_*(S^{q-1},\Z)\otimes H_*(S^{n-p-q-1},\Z)$$ is generated as a $\Z$-module by : 
$$b_{p-1}\otimes 1\otimes (b_0-b_0'), 1\otimes b_{q-1}\otimes  (b_0-b_0')\ \text{and} \  b_{p-1}\otimes b_{q-1}\otimes  (b_0-b_0') .$$
It follows that :
 
\begin{proposition}If $p\neq q$ then :
\begin{itemize}
\item[1)]The integer homology of $X_{p,q}^{p+q+1}$ is given by :
 \begin{equation*}
      H_i(X_{p,q}^{p+q+1},\Z) \simeq 
        \begin{cases}
            \Z& \text{if $i =0,p, q,p+q-1$} \\
           
            0 & \text{otherwise}
        \end{cases}
    \end{equation*}
 \item[2)]The rational homology of $X_{p,q}^{p+q+1}$ is given by :
\begin{equation*}
      H_i(X_{p,q}^{p+q+1},\mathbb{Q}) \simeq 
        \begin{cases}
            \mathbb{Q}& \text{if $i =0,p, q,p+q-1$} \\
           
            0 & \text{otherwise}
        \end{cases}
    \end{equation*}

\end{itemize}
\end{proposition}
\begin{proposition}For $p=q$ we have :
\begin{itemize}
\item[1)]The integer homology of $X_{p,p}^{2p+1}$ is given by :
 \begin{equation*}
      H_i(X_{p,p}^{2p+1},\Z) \simeq 
        \begin{cases}
            \Z& \text{if $i =0,2p-1$} \\
           \Z^2 &\text{if $i=p$}\\
            0 & \text{otherwise}
        \end{cases}
    \end{equation*}
 \item[2)]The rational homology of $X_{p,p}^{2p+1}$ is given by :
\begin{equation*}
      H_i(X_{p,p}^{2p+1},\mathbb{Q}) \simeq 
        \begin{cases}
            \mathbb{Q}& \text{if $i =0,2p-1$} \\
            \Q^2&  \text{if $i=p$}\\
            0 & \text{otherwise}
        \end{cases}
    \end{equation*}

\end{itemize}
\end{proposition}
We will denote by $a_i$ the antipod of $S^i$.  
\begin{lemma}
\begin{itemize}
\item[1)] $$(a_{p-1})_*\otimes (a_{q-1})_*\otimes (a_0)_*(b_{p-1}\otimes 1\otimes  (b_0-b_0'))=(-1)^{p+1}b_{p-1}\otimes 1\otimes  (b_0-b_0').$$
\item[2)]$$(a_{p-1})_*\otimes (a_{q-1})_*\otimes (a_0)_*(1\otimes b_{q-1}\otimes  (b_0-b_0'))=(-1)^{q+1}1\otimes b_{q-1}\otimes  (b_0-b_0').$$
\item[3)]$$(a_{p-1})_*\otimes (a_{q-1})_*\otimes (a_0)_*(b_{p-1}\otimes b_{q-1}\otimes  (b_0-b_0'))=(-1)^{p+q}b_{p-1}\otimes b_{q-1}\otimes (b_0-b_0').$$
\end{itemize}
\end{lemma}
\begin{proof}
We use that : for $k>1$, the degree of $a_k$ is $(-1)^{k+1}$, $a_k(1)=1$, $a_0(b_0)=b_0'$ and $a_0(b_0')=b_0$.
\end{proof}
\begin{theorem}For $p\neq q$, 
$H_0(Q_{p,q}^{p+q+1},\Q)=\Q$, $H_{p}(Q_{p,q}^{p+q+1},\Q)=\Q$ if $p$ is odd, $H_{q}(Q_{p,q}^{p+q+1},\Q)=\Q$ if $q$ is odd, $H_{p+q-1}(Q_{p,q}^{p+q+1},\Q)=\Q$ if $p+q=n-1$ is even and $H_i(Q_{p,q}^n,\Q)=0$ otherwise.
\end{theorem}
\begin{proof}
We apply the same reasoning as in the proof of theorem \ref{rational} and we use the previous lemma to find the invariants.
 \end{proof}
For $p=q$ :
\begin{theorem}
$H_0(Q_{p,p}^{2p+1},\Q)=\Q$, $H_{p}(Q_{p,p}^{2p+1},\Q)=\Q^2$ if $p$ is odd, $H_{2p-1}(Q_{p,p}^{2p-1},\Q)=\Q$ and $H_i(Q_{p,p}^n,\Q)=0$ otherwise.
\end{theorem}
\begin{proof}
The proof is exactly the same as for the previous theorem.
\end{proof}

\subsubsection{The case $p$ or $q$ (not both) equal to $1$ and $n-p-q >1$}\label{S223}
By the result of the previous section $X_{p,q}^n$ is homeomorphic to the join $(S^{p-1}\times S^{q-1})\star S^{n-p-q-1}$. Hence, $X_{p,q}^n$ and $X_{q,p}^n$ are homeomorphic. Therfore we only consider the case $p=1$. \\
Since $X_{1,q}^n$ is homeomorphic to the join $(S^0\times S^{q-1})\star S^{n-q-2}$ , $H_0(X_{1,q}^n,\Z)=\Z$. We have that $Tor(H_i(S^0\times S^{q-1},\Z),H_j(S^{n-q-2},\Z))=0$ for all $i,j$. Hence, by theorem \ref{Hjo} :
\begin{equation}\label{iso}
 H_{*+1}(X_{1,q}^n,\Z) \simeq Ker(i_{S^0\times S^{q-1}})\otimes Ker(i_{S^{n-q-2}}),
\end{equation}
where $i_Z$ is the map induces by the inclusion of the space $Z$ into the cone $C_Z$ which is a contractible space. For $i>0$, denote by $1$ a generator of $H_0(S^i,\Z)$ and by $b_i$ a generator of $H_i(S^i,\Z)$ and denote by $b_0$ and $b_0'$ two generators of $H_*(S^0,\Z)=H_0(S^0,\Z)\simeq \Z^2$ such that $b_0$ and $b_0'$ correpond to two simplices. 
Using the Kunneth isomorphism $H_*(S^0\times S^{q-1},\Z)\simeq H_*(S^0,\Z)\otimes H_*(S^{q-1},\Z),$ we have that :
\begin{itemize}
\item[1)] $Ker(i_{S^0\times S^{q-1}})$ is generated as a $\Z$-module by $(b_{0}-b_0')\otimes 1,b_0\otimes b_{q-1}$ and $ b_0'\otimes b_{q-1}$.
\item[2)] $Ker(i_{S^{n-q-2}})$ is generated as a $\Z$-module by $b_{n-q-2}$.
\end{itemize}
In particular, $H_{*+1}(X_{1,q}^n,\Z)$ viewed as a subspace of $$H_*(S^0,\Z)\otimes H_*(S^{q-1},\Z)\otimes H_*(S^{n-q-2},\Z)$$ is generated as a $\Z$-module by : 
$$(b_{0}-b_0')\otimes 1\otimes b_{n-q-2}, b_0\otimes b_{q-1}\otimes b_{n-q-2}\ \text{and} \   b_0'\otimes b_{q-1}\otimes b_{n-q-2} .$$
It follows that :
 
\begin{proposition}
\begin{itemize}
\item[1)]The integer homology of $X_{1,q}^n$ is given by :
 \begin{equation*}
      H_i(X_{1,q}^n,\Z) \simeq 
        \begin{cases}
            \Z^2& \text{if $i =n-2$} \\
           \Z & \text{if $i=0,n-q-1$}\\
            0 & \text{otherwise}
        \end{cases}
    \end{equation*}
 \item[2)]The rational homology of $X_{1,q}^n$ is given by :
\begin{equation*}
      H_i(X_{1,q}^n,\mathbb{Q}) \simeq 
        \begin{cases}
             \Q^2& \text{if $i =n-2$} \\
           \Q & \text{if $i=0,n-q-1$}\\
            0 & \text{otherwise}
        \end{cases}
    \end{equation*}

\end{itemize}
\end{proposition}

We will denote by $a_i$ the antipod of $S^i$.  
\begin{lemma}
\begin{itemize}
\item[1)] $$(a_0)_*\otimes (a_{q-1})_*\otimes (a_{n-q-2})_*((b_0-b_0')\otimes 1\otimes b_{n-q-2})=(-1)^{n-q}(b_0-b_0')\otimes 1\otimes b_{n-q-2}.$$
\item[2)]$$(a_0)_*\otimes (a_{q-1})_*\otimes (a_{n-q-2})_*(b_0\otimes b_{q-1}\otimes b_{n-q-2})=(-1)^{n-1}b_0'\otimes b_{q-1}\otimes b_{n-q-2}.$$
\item[3)]$$(a_0)_*\otimes (a_{q-1})_*\otimes (a_{n-q-2})_*(b_0'\otimes b_{q-1}\otimes b_{n-q-2})=(-1)^{n-1}b_0\otimes b_{q-1}\otimes b_{n-q-2}.$$
\end{itemize}
\end{lemma}
\begin{proof}
We use that : for $k>1$, the degree of $a_k$ is $(-1)^{k+1}$, $a_k(1)=1$, $a_0(b_0)=b_0'$ and $a_0(b_0')=b_0$.
\end{proof}
\begin{theorem}
$H_0(Q_{1,q}^n,\Q)=\Q$, $H_{n-q-1}(Q_{1,q}^n,\Q)=\Q$ if $n-q$ is even, , $H_{n-2}(Q_{1,q}^n,\Q)=\Q$ and $H_i(Q_{1,q}^n,\Q)=0$ otherwise.
\end{theorem}
\begin{proof}
We apply the same reasoning as in the proof of theorem \ref{rational} and we use the previous lemma to find the invariants.
 \end{proof}

\subsubsection{The case $p$ or $q$ (not both) equal to $1$ and $n-p-q =1$}\label{S224}
By the result of the previous section $X_{p,q}^n$ is homeomorphic to the join $(S^{p-1}\times S^{q-1})\star S^{n-p-q-1}$. Hence, $X_{p,q}^n$ and $X_{q,p}^n$ are homeomorphic. Therfore we only consider the case $p=1$ (not $q=1$). \\
Since $X_{1,q}^{q+2}$ is homeomorphic to the join $(S^0\times S^{q-1})\star S^0$ , $H_0(X_{1,q}^{q+2},\Z)=\Z$. We have that $Tor(H_i(S^0\times S^{q-1},\Z),H_j(S^0,\Z))=0$ for all $i,j$. Hence, by theorem \ref{Hjo} :
\begin{equation*}
 H_{*+1}(X_{1,q}^{q+2},\Z) \simeq Ker(i_{S^0\times S^{q-1}})\otimes Ker(i_{S^0}),
\end{equation*}
where $i_Z$ is the map induces by the inclusion of the space $Z$ into the cone $C_Z$ which is a contractible space. For $i>0$, denote by $1$ a generator of $H_0(S^i,\Z)$ and by $b_i$ a generator of $H_i(S^i,\Z)$ and denote by $b_0$ and $b_0'$ two generators of $H_*(S^0,\Z)=H_0(S^0,\Z)\simeq \Z^2$ such that $b_0$ and $b_0'$ correpond to two simplices. 
Using the Kunneth isomorphism $H_*(S^0\times S^{q-1},\Z)\simeq H_*(S^0,\Z)\otimes H_*(S^{q-1},\Z),$ we have that :
\begin{itemize}
\item[1)] $Ker(i_{S^0\times S^{q-1}})$ is generated as a $\Z$-module by $(b_{0}-b_0')\otimes 1,b_0\otimes b_{q-1}$ and $ b_0'\otimes b_{q-1}$.
\item[2)] $Ker(i_{S^0})$ is generated as a $\Z$-module by $b_0-b_0'$.
\end{itemize}
In particular, $H_{*+1}(X_{1,q}^{q+2},\Z)$ viewed as a subspace of $$H_*(S^0,\Z)\otimes H_*(S^{q-1},\Z)\otimes H_*(S^0,\Z)$$ is generated as a $\Z$-module by : 
$$(b_{0}-b_0')\otimes 1\otimes (b_0-b_0'), b_0\otimes b_{q-1}\otimes (b_0-b_0')\ \text{and} \   b_0'\otimes b_{q-1}\otimes (b_0-b_0').$$
It follows that :
 
\begin{proposition}
\begin{itemize}
\item[1)]The integer homology of $X_{1,q}^{q+2}$ is given by :
 \begin{equation*}
      H_i(X_{1,q}^{q+2},\Z) \simeq 
        \begin{cases}
            \Z^2& \text{if $i =q$} \\
           \Z & \text{if $i=0,1$}\\
            0 & \text{otherwise}
        \end{cases}
    \end{equation*}
 \item[2)]The rational homology of $X_{1,q}^{q+2}$ is given by :
\begin{equation*}
      H_i(X_{1,q}^{q+2},\mathbb{Q}) \simeq 
        \begin{cases}
           \Q^2& \text{if $i =q$} \\
           \Q& \text{if $i=0,1$}\\
            0 & \text{otherwise}
        \end{cases}
    \end{equation*}

\end{itemize}
\end{proposition}

We will denote by $a_i$ the antipod of $S^i$.  
\begin{lemma}
\begin{itemize}
\item[1)] $$(a_0)_*\otimes (a_{q-1})_*\otimes (a_0)_*((b_{0}-b_0')\otimes 1\otimes (b_0-b_0'))=(b_{0}-b_0')\otimes 1\otimes (b_0-b_0').$$
\item[2)]$$(a_0)_*\otimes (a_{q-1})_*\otimes (a_0)_*( b_0\otimes b_{q-1}\otimes (b_0-b_0'))=(-1)^{q+1} b_0'\otimes b_{q-1}\otimes (b_0-b_0').$$
\item[3)]$$(a_0)_*\otimes (a_{q-1})_*\otimes (a_0)_*( b_0'\otimes b_{q-1}\otimes (b_0-b_0'))=(-1)^{q+1}b_0\otimes b_{q-1}\otimes (b_0-b_0').$$
\end{itemize}
\end{lemma}
\begin{proof}
We use that : for $k>1$, the degree of $a_k$ is $(-1)^{k+1}$, $a_k(1)=1$, $a_0(b_0)=b_0'$ and $a_0(b_0')=b_0$.
\end{proof}
\begin{theorem}
$H_0(Q_{1,q}^{q+2},\Q)=\Q$, $H_1(Q_{1,q}^{q+2},\Q)=\Q$, $H_{q}(Q_{1,q}^{q+2},\Q)=\Q$ and $H_i(Q_{1,q}^{q+2},\Q)=0$ otherwise.
\end{theorem}
\begin{proof}
We apply the same reasoning as in the proof of theorem \ref{rational} and we use the previous lemma to find the invariants.
 \end{proof}

\subsubsection{The case $p=q=1$ and $n-p-q >1$}\label{S225}
Since $X_{1,1}^n$ is homeomorphic to the join $(S^0\times S^0)\star S^{n-3}$ , $H_0(X_{1,1}^n,\Z)=\Z$. We have that $Tor(H_i(S^0\times S^0,\Z),H_j(S^{n-3},\Z))=0$ for all $i,j$. Hence, by theorem \ref{Hjo} :
\begin{equation*}
 H_{*+1}(X_{1,1}^n,\Z) \simeq Ker(i_{S^0\times S^0})\otimes Ker(i_{S^{n-3}}),
\end{equation*}
where $i_Z$ is the map induces by the inclusion of the space $Z$ into the cone $C_Z$ which is a contractible space. For $i>0$, denote by $1$ a generator of $H_0(S^i,\Z)$ and by $b_i$ a generator of $H_i(S^i,\Z)$ and denote by $b_0$ and $b_0'$ two generators of $H_*(S^0,\Z)=H_0(S^0,\Z)\simeq \Z^2$ such that $b_0$ and $b_0'$ correpond to two simplices. 
Using the Kunneth isomorphism $H_*(S^0\times S^0,\Z)\simeq H_*(S^0,\Z)\otimes H_*(S^0,\Z),$ we have that :

\begin{itemize}
\item[1)] $H_*(S^0\times S^0,\Z)$ is generated by $c_1=b_0\otimes b_0,c_2=b_0\otimes b_0', c_3=b_0'\otimes b_0$ and $c_4=b_0'\otimes b_0'$.
\item[1)] $Ker(i_{S^0\times S^0})$ is generated as a $\Z$-module by $c_1-c_2, c_2-c_3$ and $ c_3-c_4$.
\item[2)] $Ker(i_{S^{n-3}})$ is generated as a $\Z$-module by $b_{n-3}$.
\end{itemize}
In particular, $H_{*+1}(X_{1,1}^n,\Z)$ viewed as a subspace of $$H_*(S^0,\Z)\otimes H_*(S^0,\Z)\otimes H_*(S^{n-3},\Z)$$ is generated as a $\Z$-module by : 
$$(c_1-c_2) \otimes b_{n-3},(c_2-c_3)\otimes b_{n-3}\ \text{and} \  (c_3-c_4)\otimes b_{n-3} .$$
It follows that :
 
\begin{proposition}
\begin{itemize}
\item[1)]The integer homology of $X_{1,1}^n$ is given by :
 \begin{equation*}
      H_i(X_{1,1}^n,\Z) \simeq 
        \begin{cases}
            \Z^3& \text{if $i =n-2$} \\
           \Z & \text{if $i=0$}\\
            0 & \text{otherwise}
        \end{cases}
    \end{equation*}
 \item[2)]The rational homology of $X_{1,1}^n$ is given by :
\begin{equation*}
      H_i(X_{1,1}^n,\mathbb{Q}) \simeq 
        \begin{cases}
             \Q^3& \text{if $i =n-2$} \\
           \Q & \text{if $i=0$}\\
            0 & \text{otherwise}
        \end{cases}
    \end{equation*}

\end{itemize}
\end{proposition}

We will denote by $a_i$ the antipod of $S^i$. 
\begin{lemma}
\begin{itemize}
\item[1)] $$(a_0)_*\otimes (a_0)_*\otimes (a_{n-3})_*((c_1-c_2) \otimes b_{n-3})=(-1)^{n-1}(c_3-c_4)\otimes b_{n-3}.$$
\item[2)]$$(a_0)_*\otimes (a_0)_*\otimes (a_{n-3})_*((c_2-c_3)\otimes b_{n-3})=(-1)^{n-1}(c_2-c_3)\otimes b_{n-3}.$$
\item[3)]$$(a_0)_*\otimes (a_0)_*\otimes (a_{n-3})_*((c_3-c_4)\otimes b_{n-3} )=(-1)^{n-1}(c_1-c_2) \otimes b_{n-3}.$$
\end{itemize}
\end{lemma}
\begin{proof}
We use that : for $k>1$, the degree of $a_k$ is $(-1)^{k+1}$, $a_k(1)=1$, $a_0(b_0)=b_0'$ and $a_0(b_0')=b_0$.
\end{proof}
\begin{theorem}
$H_0(Q_{1,1}^n,\Q)=\Q$, $H_{n-2}(Q_{1,1}^n,\Q)=\Q^2$ if $n-1$ is even, $H_{n-2}(Q_{1,1}^n,\Q)=\Q$ if $n-1$ is odd and $H_i(Q_{1,1}^n,\Q)=0$ otherwise.
\end{theorem}
\begin{proof}
We apply the same reasoning as in the proof of theorem \ref{rational} and we use the previous lemma to find the invariants.
 \end{proof}

\subsubsection{The case $p=q=1$ and $n-p-q =1$}\label{S226}
Since $X_{1,1}^3 $ is homeomorphic to the join $(S^0\times S^0)\star S^{0}$ , $H_0(X_{1,1}^3,\Z)=\Z$. We have that $Tor(H_i(S^0\times S^0,\Z),H_j(S^0,\Z))=0$ for all $i,j$. Hence, by theorem \ref{Hjo} :
\begin{equation*}
 H_{*+1}(X_{1,1}^3,\Z) \simeq Ker(i_{S^0\times S^0})\otimes Ker(i_{S^0}),
\end{equation*}
where $i_Z$ is the map induces by the inclusion of the space $Z$ into the cone $C_Z$ which is a contractible space. Denote by $b_0$ and $b_0'$ two generators of $H_*(S^0,\Z)=H_0(S^0,\Z)\simeq \Z^2$ such that $b_0$ and $b_0'$ correpond to two simplices. 
Using the Kunneth isomorphism $H_*(S^0\times S^0,\Z)\simeq H_*(S^0,\Z)\otimes H_*(S^0,\Z),$ we have that :

\begin{itemize}
\item[1)] $H_*(S^0\times S^0,\Z)$ is generated by $c_1=b_0\otimes b_0,c_2=b_0\otimes b_0', c_3=b_0'\otimes b_0$ and $c_4=b_0'\otimes b_0'$.
\item[1)] $Ker(i_{S^0\times S^0})$ is generated as a $\Z$-module by $c_1-c_2, c_2-c_3$ and $ c_3-c_4$.
\item[2)] $Ker(i_{S^0})$ is generated as a $\Z$-module by $b_0-b_0'$.
\end{itemize}
In particular, $H_{*+1}(X_{1,1}^n,\Z)$ viewed as a subspace of $$H_*(S^0,\Z)\otimes H_*(S^0,\Z)\otimes H_*(S^{n-3},\Z)$$ is generated as a $\Z$-module by : 
$$(c_1-c_2) \otimes(b_0-b_0'),(c_2-c_3)\otimes(b_0-b_0')\ \text{and} \  (c_3-c_4)\otimes (b_0-b_0') .$$
It follows that :
 
\begin{proposition}
\begin{itemize}
\item[1)]The integer homology of $X_{1,1}^3$ is given by :
 \begin{equation*}
      H_i(X_{1,1}^3,\Z) \simeq 
        \begin{cases}
            \Z^3& \text{if $i =1$} \\
           \Z & \text{if $i=0$}\\
            0 & \text{otherwise}
        \end{cases}
    \end{equation*}
 \item[2)]The rational homology of $X_{1,1}^3$ is given by :
\begin{equation*}
      H_i(X_{1,1}^3,\mathbb{Q}) \simeq 
        \begin{cases}
             \Q^3& \text{if $i =1$} \\
           \Q & \text{if $i=0$}\\
            0 & \text{otherwise}
        \end{cases}
    \end{equation*}

\end{itemize}
\end{proposition}

We will denote by $a_i$ the antipod of $S^i$. 
\begin{lemma}
\begin{itemize}
\item[1)] $$(a_0)_*\otimes (a_0)_*\otimes (a_0)_*((c_1-c_2) \otimes (b_0-b_0'))=(c_3-c_4)\otimes(b_0-b_0').$$
\item[2)]$$(a_0)_*\otimes (a_0)_*\otimes (a_0)_*((c_2-c_3)\otimes (b_0-b_0'))=(c_2-c_3)\otimes(b_0-b_0').$$
\item[3)]$$(a_0)_*\otimes (a_0)_*\otimes (a_0)_*((c_3-c_4)\otimes(b_0-b_0') )=(c_1-c_2) \otimes (b_0-b_0').$$
\end{itemize}
\end{lemma}
\begin{proof}
We use that  $a_0(b_0)=b_0'$ and $a_0(b_0')=b_0$.
\end{proof}
\begin{theorem}
$H_0(Q_{1,1}^3,\Q)=\Q$, $H_{1}(Q_{1,1}^3,\Q)=\Q^2$  and $H_i(Q_{1,1}^3,\Q)=0$ otherwise.
\end{theorem}
\begin{proof}
We apply the same reasoning as in the proof of theorem \ref{rational} and we use the previous lemma to find the invariants.
 \end{proof}

\section{Homology of degenerate quadrics over $\Z/2\Z$}
As in the previous section, for $p,q \geq 1$ and $n\geq p+q$, $Q_{p,q}^n$ is the real locus in $\RP^{n-1}$ of the quadratic form $\mathbf{q} : \R^n \to \R,(x_1,\dots,x_n)\mapsto x_1^2+\cdots +x_p^2-x_{p+1}^2-\cdots -x_{p+q}^2$. \\\\
In this section, we compute the $\ZZ$-homology groups of $Q_{p,q}^n$.\\\\
The quadrics $Q_{p,q}^n$ and $Q_{q,p}^n$ are homeomorphic. Hence, we will assume that $q\geq p$. We note that $Q_{p,q}^{p+q}$ is empty if $p$ or $q$ is zero.
\begin{theorem}\cite{ST41}\label{St}
For $q\geq p>0$, the homology with $\Z/2\Z$ coefficients of  $Q_{p,q}^{p+q}$  is isomorphic to $H_*(S^{q-1},\Z/2\Z)\otimes H_*(\RP^{p-1},\Z/2\Z)$.
\end{theorem}
\begin{proof}
As proved by Steenrod and Tucker the projection of $Q_{p,q}^{p+q}$ onto $\RP^{p-1}$ obtained by projecting the first $p$ coordinates is a sphere bundle with fiber $S^{q-1}$. Hence we have a fibration $S^{q-1}\to Q_{p,q}^{p+q} \to \RP^{p-1}$. If $p=q=1$ the result of the theorem holds since $Q_{p,q}$ is reduced to two points. For $q\neq 1$, the only graded automorphism of the graded space $H_*(S^{q-1},\Z/2\Z)$ is the identity, hence $\pi_1(\RP^{p-1})$ acts trivially in the $\Z/2\Z$-homology of the fiber and we have a spectral sequence of the fibration with coefficient in $\Z/2\Z$ with second page given by $H_*(\RP^{p-1},\Z/2\Z) \otimes H_*(S^{q-1},\Z/2\Z)$. Now, since $q-1\geq p-1$ this spectral sequence collapses at the second page and we get that $H_*(Q_{p,q}^{p+q},\Z/2\Z)\simeq H_*(S^{q-1},\Z/2\Z)\otimes H_*(\RP^{p-1},\Z/2\Z)$.  
\end{proof}

For now on we assume that $n>p+q$. For $k\leq m$, let $U_k^m$ denote the open chart $$\{ [x_1:\cdots :x_{k-1}:1:x_{k+1}:\cdots x_m] \in \RP^{m-1} \}$$ of $\RP^{m-1}$ and let $D$ denote the closed subset of $\RP^{n-1}$ : $$ D=\{ [0:\cdots :0:x_{p+q+1}:\cdots : x_n] \in \RP^{n-1} \}.$$
$D$ is homeomorphic to $\RP^{n-p-q-1}$ and is contained in $Q_{p,q}^n$. Note that if $p$ or $q$ are zero, then $Q_{p,q}^n$ is homeomorphic to $D$. Hence, we assume in the following that $p,q\geq 1$.We have well defined map $\pi : \RP^{n-1}\setminus D\to \RP^{p+q-1}$ given by : $$\pi([x_1:\cdots x_n])=[x_1:\cdots :x_{p+q}]$$
\begin{proposition}\label{cov}
Let $f:X\to Y$ be a map between topological spaces and $(U_i)_{i\in I}$ be an open cover of $X$. If the restrictions $f_{\vert U_i}$ of $f$ to the open sets $U_i$ are continous for all $i\in I$, then $f$ is continuous.
\end{proposition}
\begin{proof}
Let $V$ be an open set of $Y$, $f^{-1}(V)=\cup_{i \in I} f_{\vert U_i}^{-1}(V) $. If $f_{\vert U_i}$ is continious then $f_{\vert U_i}^{-1}(V) $ is open in $U_i$ and hence open in $X$ since $U_i$ is open. This proves that, under the hypothesis in the proposition, the preimage of a open set is the union of open sets and therfore $f$ is continuous.
\end{proof}
\begin{corollary}
The map $\pi : \RP^{n-1}\setminus D\to \RP^{p+q-1}$ is continuous. 
\end{corollary}
\begin{proof}
The family $(U_k^n)_{k\leq p+q }$ is an open cover of $ \RP^{n-1}\setminus D$. The map $f_{\vert U_k^n}$ maps $U_k^n$ onto $U_k^{p+q}$ and is a "polynomial" mapping between the two charts. Hence, $f_{\vert U_k^n}$ is contiuous and the corollary follows from proposition.
\end{proof}
\begin{proposition}
The map $\pi : \RP^{n-1}\setminus D\to  \RP^{p+q-1}$ defines a real vector bundle of rank $n-p-q$. 
\end{proposition}
\begin{proof}
$\pi$ is a continous surjection. Now take $x \in \RP^{p+q-1}$ with coordinates $[x_1:\cdots : x_{p+q}]$. The fiber over $x$ for $\pi$ is the set :$$F_x=\{ [x_1:\cdots :x_{p+q} : \lambda_1:\cdots :\lambda_{n-p-q}] \in \RP^n \vert (\lambda_1 ,\dots,\lambda_{n-p-q})\in \R^{n-p-q}\}.$$ The map $F_x \to \R^{n-p-q}$ given by :$$ [x_1:\cdots :x_{p+q} : \lambda_1:\cdots :\lambda_{n-p-q}] \mapsto  (\lambda_1 ,\dots,\lambda_{n-p-q})$$ is well defined and is a bijection, hence we can transfer the vector space structure of $\R^{n-p-q}$ onto $F_x$. The structure obtained does not depend on the coordinates of $x$. We have show that the fibers of the map $\pi$ have the structure of real $n-p-q$ dimensional vector space.   
A point $x\in \RP^{p+q-1}$ lies in some $U_k^{p+q}$ and $\pi^{-1}(U_k^{p+q})=U_k^{n}$. We have an natural isomorphism $ \phi_k : U_k^{p+q}\times R^{n-p+q}\to U_{k}^n$ (observe the coordinates). We have that $(\pi\circ \phi_k)(x,v)=x$ for all $x\in U_k^{p+q}$ and all $v\in \R^{n-p-q}$. On the other hand for $x\in  U_k^{p+q}$ the map $v\mapsto \phi_k(x,v)$ is an isomorphism of vector spaces between $\R^{n-p-q}$ and $F_x$. This proves the lemma.
\end{proof}
We have noted that $D\subset Q_{p,q}^n$. One can readly check that the image of $Q_{p,q}^n\setminus D$ under $\pi$ is $Q_{p,q}^{p+q}$. Hence we have the following corollary : 

\begin{corollary}The map $Q_{p,q}^n\setminus D\overset{\pi}{\to}Q_{p,q}^{p+q}$ defines a real vector bundle of rank $n-p-q$.
\end{corollary}
\begin{proposition}
The Thom space $T(Q_{p,q}^n\setminus D)$ associated to $Q_{p,q}^n\setminus D \overset{\pi}{\to}Q_{p,q}^{p+q}$ is homeomorphic to the quotient space $Q_{p,q}^n/D$.
\end{proposition}
\begin{proof}
Both spaces $T(Q_{p,q}^n\setminus D)$ and $Q_{p,q}^n/ D$ are one point compactifications of $Q_{p,q}^n\setminus D$. Hence they are homeomorphic. 
\end{proof}

\begin{proposition}\label{re}
For $k>n-p-q$, $\tilde{H}_k(Q_{p,q}^n/D,\Z/2\Z)$, where the $\tilde{H}$ is for the reduced homology, is isomorphic to $H_{k-n+p+q}(Q_{p,q}^{p+q},\Z/2\Z)$
\end{proposition}
\begin{proof}
By the Thom isomorphism theorem we have an isomorphism $$H^{k-n+p+q}(Q_{p,q}^{p+q},\Z/2\Z)\simeq \tilde{H}^k(T(Q_{p,q}^n\setminus D),\Z/2\Z) ,$$ for $k\geq n-p-q$. Since we are working over a field we have also an isomorphism between homology groups $$H_{k-n+p+q}(Q_{p,q}^{p+q},\Z/2\Z)\simeq \tilde{H}_k(T(Q_{p,q}^n\setminus D),\Z/2\Z) ,$$ for $k\geq n-p-q$. Finally, the spaces $T(Q_{p,q}^n\setminus D)$ and $Q_{p,q}^n/D$ are homeomorphic and the proposition follows.
\end{proof}

\begin{proposition}
$D$ is a deformation retract of one of its neighbehouds in $Q_{p,q}^n$.
\end{proposition}
\begin{proof}
We define the closed subset $D'$ of $\RP^{n-1}$ by : 
$$D'= \{ [x_1:\cdots :x_{p+q}:0:\cdots : 0] \in \RP^{n-1} \}.$$
Denote by  $U$ the complement of $D'$ in $\RP^{n-1}$. $U$ is open and is covered by the charts $(U_k^n)_{k>p+q}$. Now define $f: U\times [0,1] \to U$ by $$f( [x_1:\cdots :x_n],t)=[(1-t)x_1:\cdots : (1-t)x_{p+q}:x_{p+q+1}:\cdots : x_n].$$
The map $f$ is well defined and is continuous as one can check that the restrictions $f_{\vert U_k^n}$ of $f$ to $U_k^n$ are continuous for $k>p+q$ (we apply proposition \ref{cov}). The map $f$ is a deformation retraction of $U$ onto $D$. Moreover, $f(U\cap Q_{p,q}^n\times [0,1])$ is included in $U\cap Q_{p,q}^n$. Hence, $f$ restricts to a deformation retraction of $U\cap Q_{p,q}^n$ onto $D\subset U\cap Q_{p,q}^n$. This proves the proposition since $U\cap Q_{p,q}^n$ is an open neighberhood of $D$ in $Q_{p,q}^n$.
\end{proof}
\begin{corollary}
For $k\geq 0$, $H_k(Q_{p,q},D,\Z/2\Z)$ is isomorphic to $\tilde{H}_k(Q_{p,q}^n/D,\Z/2\Z)$, where $\tilde{H}$ is for the reduced homology. 
\end{corollary}

\begin{theorem}\label{ThHH}
For $l>0$, $H_{n-p-q+l}(Q_{p,q}^n,\Z/2\Z)$ is isomorphic to $H_l(Q_{p,q}^{p+q};\Z/2\Z)$.
\end{theorem}
\begin{proof}
Since $D$ is homeomorphic to $\RP^{n-p-q-1}$, $H_k(D,\Z/2\Z)=0$ for $k\geq n-p-q$. Hence the exact sequence of the pair $(Q_{p,q}^n,D)$ gives that $H_k(Q_{p,q}^n,\Z/2\Z)$ is isomorphic to $H_k(Q_{p,q}^n,D,\Z/2\Z)$ for $k> n-p-q$. By the previous corollary $H_k(Q_{p,q}^n,D,\Z/2\Z)\simeq \tilde{H}_k(Q_{p,q}^n/D,\Z/2\Z)$. Hence for $k> n-p-q$, we have that $H_k(Q_{p,q}^n,\Z/2\Z) \simeq\tilde{H}_k(Q_{p,q}^n/D,\Z/2\Z)$. Applying proposition \ref{re}, we get that $H_k(Q_{p,q}^n,\Z/2\Z) \simeq H_{k-n+p+q}(Q_{p,q}^{p+q},\Z/2\Z)$, for $k> n-p-q$. The proposition follows by setting $k=n-p-q +l$.
\end{proof}

Combining the last theorem with Steenrod and Tucker result (theorem \ref{St}) we get

\begin{corollary}\label{Cpq}
For $q>p>0$, $n>p+q$ and $k>n-p-q$, we have :
$$ H_{k}(Q_{p,q}^n,\Z/2\Z)=\begin{cases}
\Z/2\Z & \text{if $n-p-q+1\leq k \leq n-q-1$ and $n-p-1\leq k\leq n-2$}\\
0 & \text{otherwise}
\end{cases}$$
\end{corollary}

For the case $q=p>2$ :
\begin{corollary}\label{Cpp3}
For $p>2$, $n>2p$ and $k>n-2p$, we have :
$$ H_{k}(Q_{p,p}^n,\Z/2\Z)=\begin{cases}
\Z/2\Z & \text{if $n-2p+1\leq k \leq n-p-2$ and $n-p\leq k \leq n-2$}\\	
\Z/2\Z^2 & \text{if $k=n-p-1$}\\
0 & \text{otherwise}
\end{cases}$$
\end{corollary}
For the case $q=p=2$ :
\begin{corollary}\label{CCpp2}
For $p=2$, $n>2p$ and $k>n-4$, we have :
$$ H_{k}(Q_{2,2}^n,\Z/2\Z)=\begin{cases}
\Z/2\Z & \text{if   $k = n-2$}\\	
\Z/2\Z^2 & \text{if $k=n-3$}\\
0 & \text{otherwise}
\end{cases}$$
\end{corollary}

For the case $p=q=1$ :
\begin{corollary}\label{Cpq1}
For $k>n-2$, $ H_{k}(Q_{1,1}^n,\Z/2\Z)=0$ 
\end{corollary}
 We have the $\ZZ$-homology groups of $Q_{p,q}^n$ ($n>p+q$) for the degrees $k\geq n-p-q+1$. We will now compute the remaining homology groups using the cover exact sequence :
$$ \cdots \to H_k(Q_{p,q}^n,\ZZ) \to H_k(X_{p,q}^n,\ZZ) \to H_k(Q_{p,q}^n,\ZZ) \to H_{k-1}(Q_{p,q}^n,\ZZ)\to \cdots,$$
where $H_k(X_{p,q}^n,\ZZ) \to H_k(Q_{p,q}^n,\ZZ) $ is the morphism induce by the covering map and $H_k(Q_{p,q}^n,\ZZ) \to H_k(X_{p,q}^n,\ZZ)$ is the transfer morphism.
The morphism $H_0(X_{p,q}^n,\ZZ)\to H_0(Q_{p,q}^n,\ZZ)$) is an isomorphism  ($X_{p,q}^n$ is path connected and the covering map is surjective and hence the above sequence can be shortened to 
$$ \cdots \to H_1(X_{p,q}^n,\ZZ) \to  H_1(Q_{p,q}^n,\ZZ) \to H_0(Q_{p,q}^n,\ZZ)\to 0.$$
 We have computed, in the previous section, the integer homology of $X_{p,q}^n$ and it is a free $\Z$-module, hence the groups $H_k(X_{p,q}^n,\ZZ)$ are isomorphic to $H_k(X_{p,q}^n,\Z)\otimes \ZZ$. We can therefore fill the cover exact sequence using result of the previous section. One can wonder if the sequence alone is sufficent to compute the $\ZZ$ homology groups of $Q_{p,q}^n$. Unfortunatly, for $q\neq p$ the algebraic computations with the sequence alone gives rise to two possibilities for the $\ZZ$ homology. The computations of the groups $H_k(Q_{p,q}^n,\ZZ)$ for $k>n-p-q-1$ allows to encounter that problem.

\subsubsection{The case $p,q>1$}
We will assume that $p,q>1$. 

\begin{lemma}
For $p,q>1$ $(n>p+q)$ and $i\leq n-p-q$ the groups $H_i(Q_{p,q}^n,\ZZ)$ are isomorphic to $\ZZ$.
\end{lemma}
\begin{proof}
It follows from the computations of the integer homology of $X_{p,q}^n$ for $p,q>1$ in subsections \ref{S221} and \ref{S222} that $H_i(X_{p,q}^n,\ZZ)=0$ for $0<i<n-q-1$. Hence, the (shortened) cover exact sequence gives the exact sequences : 
$$ 0\to H_i(Q_{p,q}^n,\ZZ) \to H_{i-1}(Q_{p,q}^n)\to 0,$$
for $p>q$ and $1\leq i \leq n-q-2$. Since we assume that $p>1$, $n-p-q \leq n-q-2$ and we have that, for $i\leq n-p-q$, $H_i(Q_{p,q}^n, \ZZ)\simeq H_0(Q_{p,q }^n,\ZZ)\simeq \ZZ$. We have proved the lemma. 
\end{proof}
We can now give all the $\ZZ$ homology groups of $Q_{p,q}^n$, for $q\geq p>1$ $(n>p+q)$.

\begin{theorem}\label{316}
For $q>p>1$ and $n>p+q$, we have :
$$ H_{k}(Q_{p,q}^n,\Z/2\Z)=\begin{cases}
\Z/2\Z & \text{if $0\leq k \leq n-q-1$ and $n-p-1\leq k\leq n-2$}\\
0 & \text{otherwise}
\end{cases}$$
\end{theorem}
\begin{proof}
This is obtained by combining corollary \ref{Cpq} and the previous lemma.
\end{proof}

For the case $q=p$ :
\begin{theorem}
For $p>1$ and $n>2p$ we have :
$$ H_{k}(Q_{p,p}^n,\Z/2\Z)=\begin{cases}
\Z/2\Z & \text{if $0\leq k \leq n-p-2$ and $n-p\leq k \leq n-2$}\\	
\Z/2\Z^2 & \text{if $k=n-p-1$}\\
0 & \text{otherwise}
\end{cases}$$
\end{theorem}
\begin{proof}
We combine corollaries \ref{Cpp3} and \ref{CCpp2} with the previous lemma.
\end{proof}
\subsubsection{The case $q>p=1$}
We assume that $q>p$ and $p=1$. It follows from the computations in subsection \ref{S223} and \ref{S224} that :
 \begin{equation*}
      H_i(X_{1,q}^n,\ZZ) \simeq 
        \begin{cases}
            \ZZ^2& \text{if $i =n-2$} \\
           \ZZ & \text{if $i=0,n-q-1$}\\
            0 & \text{otherwise}
        \end{cases}
    \end{equation*}

\begin{lemma}
The groups $H_k(Q_{1,q}^n,\ZZ)$ are isomorphic to $\ZZ$ for $0\leq k \leq n-q-2$.
\end{lemma}
\begin{proof}
This is true $n=q+2$. For $n>q+2$, we have $H_i(X_{1,q}^n,\ZZ)=0$ for $0<i<n-q-1$, hence by the shortened exact sequence of the cover $H_k(Q_{1,q}^n,\ZZ)\simeq H_{k-1}(Q_{1,q}^n,\ZZ)$ for $1 \leq k \leq n-q-2$. The lemma follows since $H_0(Q_{1,q}^n,\ZZ)=\ZZ$.
\end{proof}
\begin{lemma}
$H_{n-q-1}(Q_{1,q}^n,\ZZ)$ is isomorphic to $\ZZ$.
\end{lemma}
\begin{proof}
We first consider the case where $q\neq  2$. In that case by corollary \ref{Cpq} $H_{n-q}(Q_{1,q}^n,\ZZ^2)=0$ and we get the following exact sequence by filling the cover exact sequence using the previous lemma and the $\ZZ$ homology of $X_{1,q}^n$ :
$$ 0\to H_{n-q-1}(Q_{1,q}^n,\ZZ)\to \ZZ \to H_{n-q-1}(Q_{1,q}^n,\ZZ) \to \ZZ \to 0 .$$
The first two morphisms on the left show that $dim( H_{n-q-1}(Q_{1,q}^n,\ZZ) ) \leq 1$. The last two morphisms on the right prove that $dim( H_{n-q-1}(Q_{1,q}^n,\ZZ) ) \geq 1$ . Therefore, $H_{n-q-1}(Q_{1,q}^n,\ZZ) \simeq \ZZ$.\\ We no consider the case $q=2$ $(n-q-1=n-3)$. Using the homology of $X_{1,q}^n$ and the previous lemma, we get the following exact sequence (from the cover exact sequence):
\begin{align*}H_{n-1}(Q_{1,q}^n,\ZZ)&\to H_{n-2}(Q_{1,q}^n,\ZZ) \to \ZZ^2 \to  H_{n-2}(Q_{1,q}^n,\ZZ) \\& \to  H_{n-3}(Q_{1,q}^n,\ZZ)\to \ZZ\to H_{n-3}(Q_{1,q}^n,\ZZ)\to \ZZ\to 0.\end{align*}
By corollary \ref{Cpq}, we have that $H_{n-1}(Q_{1,q}^n,\ZZ)=0$ and $ H_{n-2}(Q_{1,q}^n,\ZZ)=\ZZ$. Hence we have an exact sequence :
\begin{align*}0\to \ZZ \to \ZZ^2 \to  \ZZ  \to & H_{n-3}(Q_{1,q}^n,\ZZ)\to \\& \ZZ\to H_{n-3}(Q_{1,q}^n,\ZZ)\to \ZZ\to 0.\end{align*}
Considering the first $3$ morphisms to the left, we notice that we have the exact sequence :
\begin{align*}0\to H_{n-3}(Q_{1,q}^n,\ZZ)\to \ZZ\to H_{n-3}(Q_{1,q}^n,\ZZ)\to \ZZ\to 0.\end{align*}
We deduce from such sequence as in the case $q\neq 2$, that $ H_{n-3}(Q_{1,q}^n,\ZZ)= H_{n-q-1}(Q_{1,q}^n,\ZZ) \simeq \ZZ$. We have proved the lemma.
\end{proof}
\begin{theorem}
For $q>1$ and $n>1+q$ $(p=1)$,
$$ H_{k}(Q_{1,q}^n,\Z/2\Z)=\begin{cases}
\Z/2\Z & \text{if $0\leq k \leq n-q-1$ and $ k=n-2$}\\
0 & \text{otherwise}
\end{cases}$$
\end{theorem}
\begin{proof}
The theorem is obtained by combining the last two lemmas with corollary \ref{Cpq}.
\end{proof}

\subsubsection{The case $p=q=1$}
We assume that $p=q=1$.
It follows from the computations in subsection \ref{S225} and \ref{S226} that :
\begin{equation*}
      H_i(X_{1,1}^n,\ZZ) \simeq 
        \begin{cases}
             \ZZ^3& \text{if $i =n-2$} \\
           \ZZ & \text{if $i=0$}\\
            0 & \text{otherwise}
        \end{cases}
    \end{equation*}

\begin{lemma}
For $n\geq 3$ and $k\leq n-3$, $H_k(Q_{1,1}^n,\ZZ)$ is isomorphic to $\ZZ$.
\end{lemma}
\begin{proof}
This is true for $n=3$. For $n>3$, one uses the (shortened) cover exact sequence and the fact that $H_{k}(X_{1,1}^n,\ZZ)=0$ for $0<k<n-2$ as we have seen above.
\end{proof}
\begin{lemma}
For $n\geq 3$, the group $H_{n-2}(Q_{1,1}^n,\ZZ)$ is isomorphic to $\ZZ^2$
\end{lemma}
\begin{proof}
Using the (shortened) cover exact sequence, the fact that $H_{n-3}(X_{1,1}^n,\ZZ)=0$ for the case $n>3$, that $H_{n-2}(X_{1,1}^n,\ZZ)\simeq \ZZ^3$ and that $H_{n-3}(Q_{1,1}^n,\ZZ)\simeq \ZZ$ by the previous lemma, we get the exact sequence :
$$ H_{n-1}(Q_{1,1}^n,\ZZ)\to H_{n-2}(Q_{1,1}^n,\ZZ) \to \ZZ^3\to H_{n-2}(Q_{1,1}^n,\ZZ)\to \ZZ \to 0.$$
By corollary \ref{Cpq1} $H_{n-1}(Q_{1,1}^n,\ZZ)=0$ and hence we have  :
$$ 0\to H_{n-2}(Q_{1,1}^n,\ZZ) \to \ZZ^3\to H_{n-2}(Q_{1,1}^n,\ZZ)\to \ZZ \to 0.$$
This implies that : 
\begin{align*}
\ZZ^3&\simeq H_{n-2}(Q_{1,1}^n,\ZZ)\oplus Im(\ZZ^3\to H_{n-2}(Q_{1,1}^n,\ZZ))\\
&\simeq H_{n-2}(Q_{1,1}^n,\ZZ)\oplus Ker(H_{n-2}(Q_{1,1}^n,\ZZ)\to \ZZ).
\end{align*}
Hence : $$3=2dim(H_{n-2}(Q_{1,1}^n,\ZZ))-1$$
where, the dimension is the dimension as a $\ZZ$ vector space. Therfore $$dim(H_{n-2}(Q_{1,1}^n,\ZZ))=2.$$
\end{proof}
\begin{theorem}
For $n>3$ :
$$H_k(Q_{1,1}^n,\ZZ)= \begin{cases}
\ZZ^2 &\text{if $k=n-2$} \\
\ZZ & \text{if $0\leq k \leq n-3$}\\
0 & \text{otherwise}
 \end{cases}$$
\end{theorem}
\begin{proof}
The theorem follows from corollary \ref{Cpq1} and the last two lemmas.
\end{proof}
\section{Integer homology of degenerate quadrics} 
In this section, we show that the integer homology of $Q_{p,q}^n$ is determined by the Betti numbers of $Q_{p,q}^n$ over $\Q$ and over $\ZZ$ wich we have obtained in the last two sections. We compute the integer homology of $Q_{p,q}^n$ for the case $q>p>1$, $n>p+q+1$ and where $p,q$ and $n$ are even. The other cases can also be computed; we do not consider them to avoid more computations.   

\begin{proposition}
Let $\tilde{X}$ be a non trivial two sheeted covering of $X$. If $H_*(\tilde{X},\Z)$ is a free $\Z$-module of finite rank and $H_*(X,\ZZ)$ is of finite dimension then :
\begin{itemize}
\item[1)] $H_k(X,\ZZ)=\Z^{m_k}\oplus \ZZ^{l_k}$ for some positive integers $l_k$ and $m_k$. 
\item[2)] For $k \geq 0$ :
        $$m_k=b_k(X,\Q), \quad l_0=0,\quad  l_{k+1}=b_{k+1}(X,\ZZ)-b_{k+1}(X,\Q)-l_{k},$$
where $b_k(X,\Q)$ is the $k$-th Betti number of $X$ over the rationals and $b_k(X,\ZZ)$ is the $k$-th Betti number of $X$ over $\Z/2\Z$
\item[3)] The integer homology of $X$ is determined by the rational homology and $\Z/2\Z$ homology of $X$.
\end{itemize}
\end{proposition}
\begin{proof}
Let $T_*$ be the transfer morphism $H_k(X,\Z)\to H_n(\tilde{X},\Z)$ and denote by $\pi$ the covering map $\tilde{X}\to X$. The morphism $\pi_*T_*$ correspond to multiplication by $2$. Hence $Ker(T_*)$ is a subgroup of the $2$-torsion subgroup $G_2$ of $H_k(X,\Z)$. Since $H_*(\tilde{X},\Z)$ is free of finite rank, $Im(T_*)\simeq H_*(X,\Z)/ Ker(T_*)$ is free of finite rank and $Ker(T_*)=G_2$. Morevoer, $H_*(X,\Z)\simeq Im(T_*)\oplus G_2 \simeq \Z^{m_k} \oplus G_2$ for some positive integer $m_k$. Know by the universal coefficient theorem $H_*(X,\ZZ)$ gives an upper bound on the dimenision of $H_*(X,\Z) \otimes \ZZ \simeq \ZZ^{m_k}\oplus G_2$. Since we assume $H_*(X,\ZZ)$ is finite dimensioal we have that $G_2\simeq \ZZ^{l_k}$ for some positive integer $l_k$. We have proved $1)$.  $2)$ is application of the universal coefficient theorem. $3)$ is implied by $2)$ and $1)$.
\end{proof}
The proposition applies to the cover $X_{p,q}^n\to Q_{p,q}^n$. This shows that the computations of the rational homology and $\ZZ$ homology of $Q_{p,q}^n$ accomplished in the previous sections determine the integer homology of $X$. In the case of $Q_{p,q}^n$, we can easly compute the integer homology of $Q_{p,q}^n$ from the Betti numbers over the rationals and $\ZZ$. We will compute the integer homology for a general case.\\\\
We assume that $q>p>1$, $n>p+q+1$ and that  $p,q,n$ are even. By theorem \ref{rational} :
$$b_k(X,\Q)=\begin{cases}
1 &\text{if $k=0,n-q-1,n-p-1, n-2$}\\
0 & \text{otherwise},
\end{cases}$$
and by theorem \ref{316}
$$b_k(X,\ZZ)=\begin{cases}
1 &\text{if $0\leq k\leq n-q-1$ and $ n-p-1\leq k \leq n-2$}\\
0 & \text{otherwise},
\end{cases}$$
\begin{theorem}
For $q>p>1$, $n>p+q+1$ and $p,q,n$ even :
\begin{itemize}
\item[1)] $H_k(X,\Z)=\Z$ for $k=0,n-q-1,n-p-1,n-2$ 
\item[2)] For $0<k<n-q-1$, $$H_k(X,\Z)=\begin{cases}\ZZ &\text{if $k$ is odd}\\ 0& \text{if $k$ is even} \end{cases}.$$
\item[3)] For $n-q-1<k<n-p-1$, $H_k(X,\Z)=0$.
\item[4)] For $n-p-1<k<n-2$, 
$$H_k(X,\Z)=\begin{cases}\ZZ &\text{if $k$ is even}\\ 0& \text{if $k$ is odd} \end{cases}.$$
\item[5)] For $k>n-2$, $H_k(X,\Z)=0$.
\end{itemize}
\end{theorem}
\begin{proof}
We will use the previous proposition.
We prove $1)$. For $k=0,n-q-1,n-p-1,n-2$, $b_{k}(X,\Q)=b_k(X,\ZZ)=1$. Hence $m_k=1$ and $l_k=-l_{k-1}$ for $k>0$ and $l_0=0$. Since $l_k\geq 0$ we have $l_k=0$. This proves that $H_k(X,\Z)\simeq \Z$ for $k=0,n-q-1,n-p-1,n-2$.\\We prove 2). For the $k$ mentioned $b_k(X,\Q)=0$ and $b_k(X,\ZZ)=1$. Hence, $m_k=0$ and $l_k=1-l_{k-1}$ with $l_0=0$. Point $2)$ follows.\\ We prove $3)$. For the $k$ mentioned $b_k(X,\Q)=b_k(X,\ZZ)=0$, hence $m_k=0$ and $l_k=-l_{k-1}$. Since $l_k\geq 0$, $l_k=0$ and $H_k(X,\Z)=0$.\\
We prove $4)$ as we have done for $2)$ but by using that $l_{n-p-1}=0$ (follows from $1)$ ).\\
$5)$ is proved as we have done for $3)$.
\end{proof}
The integer homology of $Q_{p,q}^n$ in the other cases can be easly determined using the same method.

 \end{document}